\documentclass[final,1p,times]{elsarticle}
\usepackage{amsmath,amssymb,amsthm,mathrsfs}
\usepackage{geometry}
\usepackage{enumitem}
\usepackage{hyperref}
\newtheorem{thm}{Theorem}[section]

\newtheorem{rem}[thm]{Remark}
\newtheorem{prop}[thm]{Proposition}
\newtheorem{cor}[thm]{Corollary}

\newtheorem{example}[thm]{Example}
\newdefinition{defn}[thm]{Definition}
\allowdisplaybreaks

\newcommand{\BH}{\mathcal{B}(\mathcal{H})}
\newcommand{\norm}[1]{\left\lVert #1 \right\rVert}

\journal{Mathematica Slovaca}
\begin{document}
	\begin{frontmatter}
	
		\title{  Refined upper bounds for the numerical radius via weighted operator means}
		\author{Shankhadeep Mondal}
		\ead{shankhadeep.mondal@ucf.edu}
		\author{Ram Narayan Mohapatra}
		\ead{Ram.Mohapatra@ucf.edu}
        \author{Kasun Tharuka Dewage}
		\ead{KasunTharuka.Dewage@ucf.edu}

		\address{School of Mathematics, University of Central Florida, Orlando, Florida-32816}
		\cortext[shan]{Corresponding Author:shankhadeep.mondal@ucf.edu}	
		
\begin{abstract}
We establish a parameterized family of upper bounds for the numerical radius of bounded linear operators on a complex Hilbert space, based on weighted expressions involving the modulus of an operator and its adjoint. The proposed estimates encompass a known numerical-radius inequality as a particular symmetric case, while optimization over the weight parameter provides additional flexibility and can lead to sharper bounds. Under suitable structural assumptions on the associated auxiliary operators, we further derive corresponding spectral-radius estimates. The weighted approach is also extended to $2\times2$ off-diagonal operator matrices. Finally, we examine sharpness and
equality cases and provide explicit finite-dimensional examples
illustrating the obtained estimates.
\end{abstract}

		
	\end{frontmatter}
	
\section{Introduction}

Let $\mathcal H$ be a complex Hilbert space and let $\BH$ denote the algebra
of all bounded linear operators on $\mathcal H$. For $A\in\BH$, the numerical
range $W(A)$ and the numerical radius $w(A)$ are defined, respectively, by
\[
W(A)=\{\langle Ax,x\rangle:x\in\mathcal H,\ \|x\|=1\},
\qquad
w(A)=\sup\{|\lambda|:\lambda\in W(A)\}.
\]
It is well known that $w(\cdot)$ defines a norm on $\BH$ equivalent to the
operator norm and satisfies
\begin{equation}\label{eq:basic}
\frac12\|A\|\leq w(A)\leq\|A\|,
\qquad A\in\BH.
\end{equation}
Moreover, the spectral radius satisfies $r(A)\leq w(A)$, while
$w(A)=\|A\|$ whenever $A$ is normal; see, for example,
\cite{HornJohnson}.

Considerable attention has been devoted to refining \eqref{eq:basic} and
obtaining sharper upper bounds for the numerical radius. Fundamental
contributions in this direction were obtained by Kittaneh, who established
several inequalities involving $|A|=(A^*A)^{1/2}$ and $|A^*|$; see
\cite{Kittaneh1,Kittaneh2,KittanehNorm}. In particular, one has
\begin{equation}\label{eq:kittaneh}
w^2(A)
\leq
\frac12\big\||A|^2+|A^*|^2\big\|.
\end{equation}
Further refinements of numerical-radius inequalities have subsequently been
obtained by incorporating operator norms, higher powers, and additional
mixed terms. We refer, in particular, to
\cite{AbuOmarKittaneh,Dragomir,BhuniaPaulBSM,BhuniaPaulArch,
BhuniaPaulResults,BhuniaPaulLMA,Kato}
and the references therein. Numerical-radius estimates for operator matrices
have also been extensively investigated; see, for example,
\cite{Alomari,BhuniaPaulLMA}. Related aspects of numerical ranges, spectral
properties, and unitary invariance may be found in
\cite{FongHolbrook,HornJohnson}.

Of particular relevance to the present work is the following estimate: for
$r\geq1$,
\begin{equation}\label{eq:bhunia-paul}
w^{2r}(A)
\leq
\frac14\big\||A|^{2r}+|A^*|^{2r}\big\|
+
\frac12\,w\!\left(|A|^r|A^*|^r\right).
\end{equation}
This inequality combines a symmetric expression involving $|A|$ and $|A^*|$
with a numerical-radius term involving their product.

\medskip

\noindent\textbf{Motivation and contribution.}
The estimate \eqref{eq:bhunia-paul} treats $|A|$ and $|A^*|$ symmetrically.
This motivates the introduction of a parameter that allows the contributions
of these two positive operators to be weighted differently. More precisely,
for $r\geq1$ and $0\leq\theta\leq1$, we establish the estimate
\[
w^{2r}(A)
\leq
\frac14
\big\|
|A|^{4r\theta}
+
|A^*|^{4r(1-\theta)}
\big\|
+
\frac12
w\!\left(
|A|^{2r\theta}|A^*|^{2r(1-\theta)}
\right).
\]
The symmetric choice $\theta=\frac12$ recovers
\eqref{eq:bhunia-paul}. Thus, the parameter $\theta$ provides a family of
numerical-radius estimates containing the symmetric inequality as a
particular case. This additional flexibility also makes it natural to
consider optimization of the resulting upper bound with respect to
$\theta$.

We further investigate spectral-radius formulations of the weighted estimate
under suitable structural assumptions. In particular, when the relevant
auxiliary weighted product is normal, its numerical radius coincides with its
spectral radius, allowing the corresponding term to be expressed in spectral
form. The weighted approach is also extended to $2\times2$ off-diagonal
operator matrices. Finally, we study sharpness and equality behavior and
provide finite-dimensional examples illustrating both equality and strict
inequality.

\medskip

\noindent\textbf{Comparison with existing bounds.}
The main features of the weighted approach and its relation to the symmetric
estimate may be summarized as follows.

\medskip

\begin{center}
\renewcommand{\arraystretch}{1.3}
\begin{tabular}{|p{0.42\textwidth}|p{0.48\textwidth}|}
\hline
\textbf{Existing framework}
&
\textbf{Present framework}
\\
\hline
Symmetric powers $|A|^{2r}$ and $|A^*|^{2r}$
&
Weighted powers $|A|^{4r\theta}$ and
$|A^*|^{4r(1-\theta)}$
\\
\hline
Numerical-radius term
$w(|A|^r|A^*|^r)$
&
Weighted numerical-radius term
$w\!\left(|A|^{2r\theta}|A^*|^{2r(1-\theta)}\right)$
\\
\hline
Fixed symmetric choice
&
Parameterized family of bounds with
$0\leq\theta\leq1$
\\
\hline
Symmetric estimate \eqref{eq:bhunia-paul}
&
Recovered by taking $\theta=\frac12$
\\
\hline
Numerical-radius formulation
&
Spectral-radius formulation under suitable structural assumptions
\\
\hline
\end{tabular}
\end{center}

\medskip

The paper is organized as follows. In Section~2, we establish the main
weighted numerical-radius inequality and discuss optimization with respect
to the parameter $\theta$. Section~3 considers spectral-radius formulations
under appropriate structural assumptions. In Section~4, the weighted
estimates are extended to off-diagonal $2\times2$ operator matrices.
Section~5 discusses sharpness and equality cases, while Section~6 presents
finite-dimensional examples illustrating the obtained estimates.

\section{Weighted Geometric Mean Inequality}

In this section, we introduce a new refinement of numerical radius inequalities
based on \emph{weighted geometric means} of the positive operators $|A|$ and
$|A^*|$.  
The motivation for this approach comes from the observation that most existing
upper bounds for the numerical radius involve either symmetric expressions such
as $|A|^r|A^*|^r$ or convex combinations of $|A|^{2r}$ and $|A^*|^{2r}$.  
While these bounds are sharp in certain extremal cases (e.g.\ normal or
nilpotent operators), they do not fully exploit the flexibility allowed by
interpolating between $|A|$ and $|A^*|$. By introducing a weight parameter $\theta\in[0,1]$, we obtain a
parameterized family of upper bounds that contains the corresponding symmetric inequality as the special case $\theta=\frac12$.

\medskip

We begin with our main result of this section.

\begin{thm}\label{thm:main}
Let $A\in\mathcal{B}(\mathcal{H})$, $r\geq 1$, and
$0\leq \theta\leq 1$. Then
\[
w^{2r}(A)
\leq
\frac14
\left\|
|A|^{4r\theta}
+
|A^*|^{4r(1-\theta)}
\right\|
+
\frac12
w\!\left(
|A|^{2r\theta}
|A^*|^{2r(1-\theta)}
\right).
\]
\end{thm}

\begin{proof}
Let $x\in\mathcal{H}$ with $\|x\|=1$. By Kato's generalized
mixed Schwarz inequality, for every $0\leq\theta\leq1$,
\[
|\langle Ax,x\rangle|^2
\leq
\langle |A|^{2\theta}x,x\rangle
\langle |A^*|^{2(1-\theta)}x,x\rangle.
\]
Raising both sides to the power $r\geq1$, we obtain
\[
|\langle Ax,x\rangle|^{2r}
\leq
\langle |A|^{2\theta}x,x\rangle^r
\langle |A^*|^{2(1-\theta)}x,x\rangle^r.
\]
Since $|A|^{2\theta}$ and $|A^*|^{2(1-\theta)}$ are positive
operators, the McCarthy inequality gives
\[
\langle |A|^{2\theta}x,x\rangle^r
\leq
\langle |A|^{2r\theta}x,x\rangle
\]
and
\[
\langle |A^*|^{2(1-\theta)}x,x\rangle^r
\leq
\langle |A^*|^{2r(1-\theta)}x,x\rangle.
\]
Consequently,
\[
|\langle Ax,x\rangle|^{2r}
\leq
\langle |A|^{2r\theta}x,x\rangle
\langle |A^*|^{2r(1-\theta)}x,x\rangle.
\]

Set
\[
u=|A|^{2r\theta}x,
\qquad
v=|A^*|^{2r(1-\theta)}x.
\]
Since $\|x\|=1$, Buzano's inequality yields
\[
|\langle u,x\rangle\langle x,v\rangle|
\leq
\frac12
\left(
\|u\|\,\|v\|
+
|\langle u,v\rangle|
\right).
\]
Therefore,
\[
\begin{aligned}
|\langle Ax,x\rangle|^{2r}
&\leq
\frac12
\Big(
\|\,|A|^{2r\theta}x\|
\,
\|\,|A^*|^{2r(1-\theta)}x\|
\\
&\hspace{2.5cm}
+
\big|
\langle
|A|^{2r\theta}x,
|A^*|^{2r(1-\theta)}x
\rangle
\big|
\Big).
\end{aligned}
\]
Using the arithmetic--geometric mean inequality,
\[
\sqrt{ab}\leq\frac{a+b}{2},
\qquad a,b\geq0,
\]
we have
\[
\begin{aligned}
&\frac12
\|\,|A|^{2r\theta}x\|
\,
\|\,|A^*|^{2r(1-\theta)}x\|
\\
&\qquad\leq
\frac14
\left(
\langle |A|^{4r\theta}x,x\rangle
+
\langle |A^*|^{4r(1-\theta)}x,x\rangle
\right).
\end{aligned}
\]
Moreover,
\[
\begin{aligned}
&
\big|
\langle
|A|^{2r\theta}x,
|A^*|^{2r(1-\theta)}x
\rangle
\big|
\\
&\qquad=
\big|
\langle
|A^*|^{2r(1-\theta)}
|A|^{2r\theta}x,x
\rangle
\big|.
\end{aligned}
\]
Hence
\[
\begin{aligned}
|\langle Ax,x\rangle|^{2r}
&\leq
\frac14
\left\langle
\left(
|A|^{4r\theta}
+
|A^*|^{4r(1-\theta)}
\right)x,x
\right\rangle
\\
&\quad+
\frac12
\left|
\left\langle
|A^*|^{2r(1-\theta)}
|A|^{2r\theta}x,x
\right\rangle
\right|.
\end{aligned}
\]
Taking the supremum over all unit vectors $x\in\mathcal{H}$ gives
\[
w^{2r}(A)
\leq
\frac14
\left\|
|A|^{4r\theta}
+
|A^*|^{4r(1-\theta)}
\right\|
+
\frac12
w\!\left(
|A^*|^{2r(1-\theta)}
|A|^{2r\theta}
\right).
\]
Finally, since $w(T)=w(T^*)$,
\[
w\!\left(
|A^*|^{2r(1-\theta)}
|A|^{2r\theta}
\right)
=
w\!\left(
|A|^{2r\theta}
|A^*|^{2r(1-\theta)}
\right),
\]
and the desired inequality follows.
\end{proof}
\medskip

The following corollaries record some immediate consequences of
Theorem~\ref{thm:main} and its relation to the corresponding symmetric estimate.

\begin{cor}\label{cor:theta-half}
Let $A\in\BH$ and $r\geq 1$. Then
\[
w^{2r}(A)
\leq
\frac14
\norm{|A|^{2r}+|A^*|^{2r}}
+
\frac12\,
w\!\left(|A|^r|A^*|^r\right).
\]
\end{cor}

\begin{proof}
The result follows immediately from Theorem~\ref{thm:main} by choosing $\theta=\tfrac12$.

\end{proof}

\begin{rem}
For $\theta=\frac12$, Theorem~\ref{thm:main} reduces to
Corollary~\ref{cor:theta-half}, which is the corresponding
Bhunia--Paul inequality. Thus, Theorem~\ref{thm:main} provides a
one-parameter extension of this symmetric estimate.
\end{rem}

\medskip

We next obtain a simpler, but sometimes more practical, estimate by bounding the
numerical radius term.

\begin{cor}\label{cor:norm-bound}
Let $A\in\BH$, $r\geq 1$, and $0\leq\theta\leq 1$. Then
\[
w^{2r}(A)
\leq
\frac14
\norm{
|A|^{4r\theta}
+
|A^*|^{4r(1-\theta)}
}
+
\frac12
\norm{
|A|^{2r\theta}
|A^*|^{2r(1-\theta)}
}.
\]
\end{cor}

\begin{proof}
By Theorem~\ref{thm:main},
\[
w^{2r}(A)
\leq
\frac14
\norm{
|A|^{4r\theta}
+
|A^*|^{4r(1-\theta)}
}
+
\frac12
w\!\left(
|A|^{2r\theta}
|A^*|^{2r(1-\theta)}
\right).
\]
Since $w(T)\leq\norm{T}$ for every $T\in\BH$, we obtain
\[
w^{2r}(A)
\leq
\frac14
\norm{
|A|^{4r\theta}
+
|A^*|^{4r(1-\theta)}
}
+
\frac12
\norm{
|A|^{2r\theta}
|A^*|^{2r(1-\theta)}
},
\]
which proves the result.
\end{proof}

\medskip

\noindent We next consider the optimization of the weighted estimate with respect to the parameter $\theta$.

\begin{prop}\label{prop:optimized-weight}
Let $A\in\BH$ be invertible and let $r\geq1$. Define
\[
\Phi_{A,r}(\theta)
=
\frac14
\norm{
|A|^{4r\theta}
+
|A^*|^{4r(1-\theta)}
}
+
\frac12
w\!\left(
|A|^{2r\theta}
|A^*|^{2r(1-\theta)}
\right),
\qquad 0\leq\theta\leq1.
\]
Then
\[
w^{2r}(A)
\leq
\min_{0\leq\theta\leq1}\Phi_{A,r}(\theta)
\leq
\Phi_{A,r}\!\left(\frac12\right),
\]
where
\[
\Phi_{A,r}\!\left(\frac12\right)
=
\frac14\norm{|A|^{2r}+|A^*|^{2r}}
+
\frac12 w\!\left(|A|^r|A^*|^r\right).
\]
Moreover, if $\Phi_{A,r}(\theta_0)
<
\Phi_{A,r}\!\left(\frac12\right),$
for some $\theta_0\in[0,1]$, then the optimized weighted estimate is
strictly sharper than the symmetric estimate corresponding to
$\theta=\frac12$.
\end{prop}

\begin{proof}
By Theorem~\ref{thm:main}, for every $\theta\in[0,1]$, $w^{2r}(A)\leq\Phi_{A,r}(\theta).$ Since $A$ is invertible, the maps
\[
\theta\longmapsto |A|^{4r\theta},
\qquad
\theta\longmapsto |A^*|^{4r(1-\theta)}
\]
are continuous in the operator norm. Consequently,
$\Phi_{A,r}$ is continuous on the compact interval $[0,1]$ and
therefore attains its minimum. Hence
\[
w^{2r}(A)
\leq
\min_{0\leq\theta\leq1}\Phi_{A,r}(\theta).
\]
Since $\frac12\in[0,1]$,
\[
\min_{0\leq\theta\leq1}\Phi_{A,r}(\theta)
\leq
\Phi_{A,r}\!\left(\frac12\right).
\]
Substituting $\theta=\frac12$ gives
\[
\Phi_{A,r}\!\left(\frac12\right)
=
\frac14\norm{|A|^{2r}+|A^*|^{2r}}
+
\frac12w\!\left(|A|^r|A^*|^r\right).
\]
The final assertion follows immediately whenever
$\Phi_{A,r}(\theta_0)<\Phi_{A,r}(\frac12)$ for some
$\theta_0\in[0,1]$.
\end{proof}

\medskip

We conclude this section with a structural remark.

\begin{rem}
The introduction of the parameter $\theta$ provides additional flexibility in
the numerical radius estimate. In particular, the weighted terms involving
$|A|$ and $|A^*|$ allow the estimate to reflect their asymmetric contributions,
whereas the choice $\theta=\tfrac12$ recovers the corresponding symmetric
inequality. Thus, optimizing over $\theta$ may yield a sharper bound for
operators whose associated weighted terms exhibit sufficient asymmetry.
\end{rem}

\section{Spectral-Radius Formulations}

In this section, we derive spectral-radius formulations of the weighted
numerical-radius estimate obtained in the previous section. Since
$r(T)\leq w(T)$ does not in general allow the numerical-radius term in
an upper bound to be replaced by the spectral radius, additional
structural assumptions are required. In particular, when the relevant
auxiliary operator is normal, its numerical radius coincides with its
spectral radius, leading to the following formulation.

\medskip

We begin with the main interpolation inequality.

\begin{prop}\label{thm:spectral}
Let $A\in\BH$, $r\geq1$, and $0\leq\theta\leq1$. Suppose that
\[
T_{\theta,r}
=
|A|^{2r\theta}|A^*|^{2r(1-\theta)}
\]
is normal. Then
\[
w^{2r}(A)
\leq
\frac14
\norm{
|A|^{4r\theta}
+
|A^*|^{4r(1-\theta)}
}
+
\frac12\,
r\!\left(
|A|^{2r\theta}|A^*|^{2r(1-\theta)}
\right).
\]
\end{prop}

\begin{proof}
By Theorem~\ref{thm:main},
\[
w^{2r}(A)
\leq
\frac14
\norm{
|A|^{4r\theta}
+
|A^*|^{4r(1-\theta)}
}
+
\frac12\,
w(T_{\theta,r}).
\]
Since $T_{\theta,r}$ is normal, its numerical radius coincides with
its operator norm and spectral radius; that is,
\[
w(T_{\theta,r})
=
\norm{T_{\theta,r}}
=
r(T_{\theta,r}).
\]
Consequently,
\[
w^{2r}(A)
\leq
\frac14
\norm{
|A|^{4r\theta}
+
|A^*|^{4r(1-\theta)}
}
+
\frac12\,
r\!\left(
|A|^{2r\theta}|A^*|^{2r(1-\theta)}
\right),
\]
which proves the result.
\end{proof}

\medskip

Taking $\theta=\frac12$ in Proposition~\ref{thm:spectral} yields the following symmetric spectral-radius formulation.

\begin{cor}\label{cor:spectral-dominance}
Let $A\in\BH$ and $r\geq1$. Suppose that $|A|^r|A^*|^r$
is normal. Then
\[
w^{2r}(A)
\leq
\frac14\norm{|A|^{2r}+|A^*|^{2r}}
+
\frac12\,r\!\left(|A|^r|A^*|^r\right).
\]
\end{cor}

\begin{proof}
Taking $\theta=\frac12$ in Proposition~\ref{thm:spectral}, we obtain
\[
|A|^{2r\theta}|A^*|^{2r(1-\theta)}
=
|A|^r|A^*|^r.
\]
Hence
\[
w^{2r}(A)
\leq
\frac14\norm{|A|^{2r}+|A^*|^{2r}}
+
\frac12\,r\!\left(|A|^r|A^*|^r\right),
\]
as required.
\end{proof}

\begin{rem}
The normality assumption in Corollary~\ref{cor:spectral-dominance}
ensures that
\[
r\!\left(|A|^r|A^*|^r\right)
=
w\!\left(|A|^r|A^*|^r\right)
=
\norm{|A|^r|A^*|^r}.
\]
Hence, under this assumption, the numerical-radius term appearing in
Corollary~\ref{cor:theta-half} can equivalently be expressed in terms
of the spectral radius.

In particular, if $A$ is normal, then $|A|=|A^*|$, and therefore
$|A|^r|A^*|^r=|A|^{2r}$
is positive and hence normal. Moreover, if $A$ is normal, then equality holds in
Corollary~\ref{cor:spectral-dominance}.
\end{rem}

\medskip

\begin{rem}\label{rem:spectral-numerical}
For every $T\in\BH$,
\[
r(T)\leq w(T)\leq\norm{T}.
\]
However, the inequality $r(T)\leq w(T)$ alone does not permit the
numerical-radius term on the right-hand side of an upper estimate to
be replaced by the spectral radius, since such a replacement would
decrease the right-hand side. Therefore, a spectral-radius version of
Theorem~\ref{thm:main} requires an additional condition ensuring
\[
r(T)=w(T).
\]
The normality assumption in Proposition~\ref{thm:spectral} provides
one such sufficient condition.
\end{rem}
\medskip

We now compare Proposition~\ref{thm:spectral} with classical inequalities involving
the operator norm.

\begin{cor}\label{cor:classical}
Let $A\in\BH$ and $r\geq1$. Suppose that
$|A|^r|A^*|^r$ is normal. Then
\[
w^{2r}(A)
\leq
\frac14\norm{|A|^{2r}+|A^*|^{2r}}
+
\frac12\norm{|A|^r|A^*|^r}.
\]
\end{cor}

\begin{proof}
By Corollary~\ref{cor:spectral-dominance},
\[
w^{2r}(A)
\leq
\frac14\norm{|A|^{2r}+|A^*|^{2r}}
+
\frac12\,r\!\left(|A|^r|A^*|^r\right).
\]
Since $r(T)\leq\norm{T}$ for every $T\in\BH$, it follows that
\[
r\!\left(|A|^r|A^*|^r\right)
\leq
\norm{|A|^r|A^*|^r}.
\]
Therefore,
\[
w^{2r}(A)
\leq
\frac14\norm{|A|^{2r}+|A^*|^{2r}}
+
\frac12\norm{|A|^r|A^*|^r},
\]
which proves the result.
\end{proof}

\medskip

We conclude this section with a conceptual remark.

\begin{rem}
The preceding results illustrate how spectral information associated with
auxiliary operators constructed from $|A|$ and $|A^*|$ can enter numerical
radius estimates under suitable structural assumptions. In particular, when
the relevant auxiliary operator is normal, its numerical radius coincides
with its spectral radius, allowing the corresponding term in the estimate
to be expressed purely in spectral terms. This provides a useful connection
between numerical-radius estimates and the spectral properties of operators
derived from $|A|$ and $|A^*|$.
\end{rem}
\section{Operator Matrix Applications}

Numerical radius inequalities for operator matrices have attracted considerable
attention due to their applications in block operator theory, system theory, and
matrix analysis.
In particular, $2 \times 2$ operator matrices naturally arise in the study of
commutators, Jordan blocks, and coupled systems.
Several sharp bounds for the numerical radius of such matrices have been obtained
by Kittaneh, Alomari, and Bhunia--Paul, typically expressed in terms of the
operator norms or symmetric combinations of the blocks. In this section, we apply the weighted estimates developed above to
off-diagonal $2\times2$ operator matrices. This yields parameterized bounds expressed in terms of the moduli of the individual blocks.

\medskip

We begin with the basic off-diagonal operator matrix, which plays a central role
in numerical radius theory.

\begin{thm}\label{thm:block}
Let $A,B\in\BH$ and let
\[
T=
\begin{pmatrix}
0 & A\\
B & 0
\end{pmatrix}
\]
act on $\mathcal H\oplus\mathcal H$. Then, for every $r\geq1$ and
$0\leq\theta\leq1$,
\[
\begin{aligned}
w^{2r}(T)
\leq {}&
\frac14
\max\Big\{
\norm{|B|^{4r\theta}+|A^*|^{4r(1-\theta)}},
\norm{|A|^{4r\theta}+|B^*|^{4r(1-\theta)}}
\Big\} +
\frac12
\max\Big\{
w\!\left(|B|^{2r\theta}|A^*|^{2r(1-\theta)}\right),
w\!\left(|A|^{2r\theta}|B^*|^{2r(1-\theta)}\right)
\Big\}.
\end{aligned}
\]
\end{thm}

\begin{proof}
For
\[
T=
\begin{pmatrix}
0 & A\\
B & 0
\end{pmatrix},
\]
we have

\[
T^*T=
\begin{pmatrix}
B^*B & 0\\
0 & A^*A
\end{pmatrix},
\qquad
TT^*=
\begin{pmatrix}
AA^* & 0\\
0 & BB^*
\end{pmatrix}.
\]
Therefore,
\[
|T|=
\begin{pmatrix}
|B| & 0\\
0 & |A|
\end{pmatrix},
\qquad
|T^*|=
\begin{pmatrix}
|A^*| & 0\\
0 & |B^*|
\end{pmatrix}.
\]
Consequently,
\[
|T|^{4r\theta}+|T^*|^{4r(1-\theta)}
=
\begin{pmatrix}
|B|^{4r\theta}+|A^*|^{4r(1-\theta)} & 0\\
0 & |A|^{4r\theta}+|B^*|^{4r(1-\theta)}
\end{pmatrix},
\]
and
\[
|T|^{2r\theta}|T^*|^{2r(1-\theta)}
=
\begin{pmatrix}
|B|^{2r\theta}|A^*|^{2r(1-\theta)} & 0\\
0 & |A|^{2r\theta}|B^*|^{2r(1-\theta)}
\end{pmatrix}.
\]

Applying Theorem~\ref{thm:main} to $T$ gives
\[
\begin{aligned}
w^{2r}(T)
\leq {}&
\frac14
\norm{|T|^{4r\theta}+|T^*|^{4r(1-\theta)}}+
\frac12
w\!\left(
|T|^{2r\theta}|T^*|^{2r(1-\theta)}
\right).
\end{aligned}
\]
For block-diagonal operators,
\[
\norm{
\begin{pmatrix}
X&0\\0&Y
\end{pmatrix}}
=
\max\{\norm{X},\norm{Y}\},
\]
and
\[
w\!\left(
\begin{pmatrix}
X&0\\0&Y
\end{pmatrix}
\right)
=
\max\{w(X),w(Y)\}.
\]
Substituting the preceding block expressions yields the desired
inequality.
\end{proof}

\medskip

The following corollary gives a particularly clean estimate in the symmetric
case.

\begin{cor}\label{cor:symmetric}
Let $A\in\BH$ and define
\[
T=
\begin{pmatrix}
0 & A\\
A & 0
\end{pmatrix}
\quad\text{on }\mathcal H\oplus\mathcal H.
\]
Then, for every $r\geq1$,
\[
w^{2r}(T)
\leq
\frac14\norm{|A|^{2r}+|A^*|^{2r}}
+
\frac12\,w\!\left(|A|^r|A^*|^r\right).
\]
\end{cor}

\begin{proof}
Taking $B=A$ and $\theta=\frac12$ in
Theorem~\ref{thm:block}, we obtain
\[
\begin{aligned}
w^{2r}(T)
\leq {}&
\frac14
\max\left\{
\norm{|A|^{2r}+|A^*|^{2r}},
\norm{|A|^{2r}+|A^*|^{2r}}
\right\}+
\frac12
\max\left\{
w\!\left(|A|^r|A^*|^r\right),
w\!\left(|A|^r|A^*|^r\right)
\right\}.
\end{aligned}
\]
Therefore,
\[
w^{2r}(T)
\leq
\frac14\norm{|A|^{2r}+|A^*|^{2r}}
+
\frac12\,w\!\left(|A|^r|A^*|^r\right),
\]
which proves the result.
\end{proof}
\medskip

We next obtain a corresponding spectral-radius formulation under
normality assumptions on the auxiliary block operators.

\begin{prop}\label{thm:block-spectral}
Let $A,B\in\BH$, let $r\geq1$, and let $0\leq\theta\leq1$. Define
\[
T=
\begin{pmatrix}
0 & A\\
B & 0
\end{pmatrix}.
\]
Suppose that the operators $|B|^{2r\theta}|A^*|^{2r(1-\theta)}
\quad\text{and}\quad
|A|^{2r\theta}|B^*|^{2r(1-\theta)}$
are normal. Then
\[
\begin{aligned}
w^{2r}(T)
\leq {}&
\frac14
\max\Big\{
\norm{|B|^{4r\theta}+|A^*|^{4r(1-\theta)}},
\norm{|A|^{4r\theta}+|B^*|^{4r(1-\theta)}}
\Big\}
\\
&+
\frac12
\max\Big\{
r\!\left(|B|^{2r\theta}|A^*|^{2r(1-\theta)}\right),
r\!\left(|A|^{2r\theta}|B^*|^{2r(1-\theta)}\right)
\Big\}.
\end{aligned}
\]
\end{prop}

\begin{proof}
By Theorem~\ref{thm:block},
\[
\begin{aligned}
w^{2r}(T)
\leq {}&
\frac14
\max\Big\{
\norm{|B|^{4r\theta}+|A^*|^{4r(1-\theta)}},
\norm{|A|^{4r\theta}+|B^*|^{4r(1-\theta)}}
\Big\}
\\
&+
\frac12
\max\Big\{
w\!\left(|B|^{2r\theta}|A^*|^{2r(1-\theta)}\right),
w\!\left(|A|^{2r\theta}|B^*|^{2r(1-\theta)}\right)
\Big\}.
\end{aligned}
\]
By hypothesis, both auxiliary operators are normal. Hence, for each of them,
the numerical radius coincides with the spectral radius. Therefore,
\[
w\!\left(|B|^{2r\theta}|A^*|^{2r(1-\theta)}\right)
=
r\!\left(|B|^{2r\theta}|A^*|^{2r(1-\theta)}\right)
\]
and
\[
w\!\left(|A|^{2r\theta}|B^*|^{2r(1-\theta)}\right)
=
r\!\left(|A|^{2r\theta}|B^*|^{2r(1-\theta)}\right).
\]
Substituting these identities into the preceding estimate gives the desired
inequality.
\end{proof}

\medskip

We next obtain a corresponding spectral-radius formulation under
normality assumptions on the auxiliary block operators. Let $A,B\in\BH$, $r\geq1$, and $0\leq\theta\leq1$. Define
\[
S_1
=
|B|^{2r\theta}|A^*|^{2r(1-\theta)}
\quad\text{and}\quad
S_2
=
|A|^{2r\theta}|B^*|^{2r(1-\theta)}.
\]

\begin{rem}\label{prop:block-normality}

If $S_1$ and $S_2$ are normal, then
\[
r(S_j)=w(S_j)=\norm{S_j},
\qquad j=1,2.
\]
Consequently, the numerical-radius terms in
Theorem~\ref{thm:block} may be expressed equivalently in terms of
the corresponding spectral radii, yielding Theorem~\ref{thm:block-spectral}.
\end{rem}

\begin{proof}
Since $S_1$ and $S_2$ are normal, the spectral theorem gives
\[
r(S_j)=\norm{S_j},
\qquad j=1,2.
\]
Moreover, the numerical radius of a normal operator coincides with
its operator norm. Hence
\[
r(S_j)=w(S_j)=\norm{S_j},
\qquad j=1,2.
\]
Substituting these identities into the estimate of
Theorem~\ref{thm:block} gives precisely the spectral-radius
formulation stated in Theorem~\ref{thm:block-spectral}.
\end{proof}

\medskip

We now compare our results with classical operator norm bounds.

\begin{cor}\label{cor:block-classical}
Let $A,B\in\BH$ and let
\[
T=
\begin{pmatrix}
0 & A\\
B & 0
\end{pmatrix}
\]
act on $\mathcal H\oplus\mathcal H$. Then
\[
w(T)
\leq
\frac12\bigl(\norm{A}+\norm{B}\bigr).
\]
\end{cor}

\begin{proof}
Let $x\oplus y\in\mathcal H\oplus\mathcal H$ satisfy
\[
\norm{x}^2+\norm{y}^2=1.
\]
Then
\[
\begin{aligned}
\left|
\left\langle
T(x\oplus y),x\oplus y
\right\rangle
\right|
&=
\left|
\langle Ay,x\rangle+\langle Bx,y\rangle
\right|
\\
&\leq
|\langle Ay,x\rangle|
+
|\langle Bx,y\rangle|
\\
&\leq
\norm{A}\norm{x}\norm{y}
+
\norm{B}\norm{x}\norm{y}
\\
&=
\bigl(\norm{A}+\norm{B}\bigr)
\norm{x}\norm{y}.
\end{aligned}
\]
Since $2\norm{x}\norm{y}
\leq
\norm{x}^2+\norm{y}^2
=1,$
we obtain
\[
\left|
\left\langle
T(x\oplus y),x\oplus y
\right\rangle
\right|
\leq
\frac12
\bigl(\norm{A}+\norm{B}\bigr).
\]
Taking the supremum over all unit vectors
$x\oplus y\in\mathcal H\oplus\mathcal H$ gives
\[
w(T)
\leq
\frac12
\bigl(\norm{A}+\norm{B}\bigr).
\]
\end{proof}

\medskip

We conclude this section with an interpretative remark.

\begin{rem}
The preceding results show that the weighted approach extends naturally
to off-diagonal operator matrices. In particular, the parameter $\theta$
provides a family of block-operator estimates, while under suitable
normality assumptions the corresponding auxiliary numerical-radius
terms may be expressed in terms of spectral radii. This suggests possible
extensions to more general block-operator structures.
\end{rem}

\section{Sharpness and Equality Cases}

In this section, we investigate the equality cases of the inequalities obtained
in the previous sections.
Such an analysis is essential for understanding the sharpness of the bounds and
for identifying the structural constraints imposed on operators when equality
occurs. As is common in numerical radius theory, equality is closely related
to the equality conditions in the inequalities used in the proofs,
such as Kato's generalized mixed Schwarz inequality, McCarthy's
inequality, Buzano's inequality, and the arithmetic--geometric mean
inequality.

We begin by characterizing equality in the weighted geometric mean inequality.

\medskip

\begin{thm}\label{thm:equality-main}
Let $A\in\BH$ be normal and let $r\geq1$. Then equality holds in
Theorem~\ref{thm:main} for $\theta=\frac12$. More precisely,
\[
w^{2r}(A)
=
\frac14\norm{|A|^{2r}+|A^*|^{2r}}
+
\frac12\,w\!\left(|A|^r|A^*|^r\right).
\]
\end{thm}

\begin{proof}
Since $A$ is normal, we have $|A|=|A^*|.$
Moreover, for a normal operator, $w(A)=\norm{A}.$ Therefore,
\[
\frac14\norm{|A|^{2r}+|A^*|^{2r}}
=
\frac14\norm{2|A|^{2r}}
=
\frac12\norm{A}^{2r}.
\]
Also, $|A|^r|A^*|^r=|A|^{2r}$ and since $|A|^{2r}$ is positive,
\[
w\!\left(|A|^r|A^*|^r\right)
=
w(|A|^{2r})
=
\norm{|A|^{2r}}
=
\norm{A}^{2r}.
\]
Consequently,
\[
\begin{aligned}
\frac14\norm{|A|^{2r}+|A^*|^{2r}}
+
\frac12\,w\!\left(|A|^r|A^*|^r\right)
=
\frac12\norm{A}^{2r}
+
\frac12\norm{A}^{2r}
=
\norm{A}^{2r} =
w^{2r}(A).
\end{aligned}
\]
Hence equality holds in Theorem~\ref{thm:main} when
$\theta=\frac12$.
\end{proof}

\medskip

We next consider the optimization of the weighted estimate with respect to
the parameter $\theta$. Let $A\in\BH$ be invertible and, for $r\geq1$, define
\[
\Phi_{A,r}(\theta)
:=
\frac14
\norm{
|A|^{4r\theta}
+
|A^*|^{4r(1-\theta)}
}
+
\frac12
w\!\left(
|A|^{2r\theta}
|A^*|^{2r(1-\theta)}
\right),
\qquad 0\leq\theta\leq1.
\]

\begin{prop}\label{prop:rigidity}
Let $A\in\BH$ be invertible and let $r\geq1$. Then the function
$\Phi_{A,r}$ is continuous on $[0,1]$ and attains its minimum.
Consequently,
\[
w^{2r}(A)
\leq
\min_{0\leq\theta\leq1}\Phi_{A,r}(\theta)
\leq
\Phi_{A,r}\!\left(\frac12\right).
\]
\end{prop}

\begin{proof}
Since $A$ is invertible, both $|A|$ and $|A^*|$ are positive
invertible operators. By the continuous functional calculus, the maps
\[
\theta\longmapsto |A|^{4r\theta},
\qquad
\theta\longmapsto |A^*|^{4r(1-\theta)}
\]
are continuous in the operator norm. Similarly, $\theta\longmapsto
|A|^{2r\theta}|A^*|^{2r(1-\theta)}$
is norm continuous. Since both the operator norm and the numerical radius are continuous
with respect to the operator norm, it follows that
$\Phi_{A,r}$ is continuous on $[0,1]$. Hence, by compactness of
$[0,1]$, there exists $\theta_0\in[0,1]$ such that
\[
\Phi_{A,r}(\theta_0)
=
\min_{0\leq\theta\leq1}\Phi_{A,r}(\theta).
\]

On the other hand, Theorem~\ref{thm:main} gives $w^{2r}(A)\leq\Phi_{A,r}(\theta)$
for every $\theta\in[0,1]$. Therefore,
\[
w^{2r}(A)
\leq
\min_{0\leq\theta\leq1}\Phi_{A,r}(\theta).
\]
Since $\theta=\frac12$ is one of the admissible choices,
\[
\min_{0\leq\theta\leq1}\Phi_{A,r}(\theta)
\leq
\Phi_{A,r}\!\left(\frac12\right).
\]
Finally, substituting $\theta=\frac12$ gives
\[
\Phi_{A,r}\!\left(\frac12\right)
=
\frac14\norm{|A|^{2r}+|A^*|^{2r}}
+
\frac12w\!\left(|A|^r|A^*|^r\right),
\]
which completes the proof.
\end{proof}

\medskip

Next, we consider equality in the spectral-radius estimate.

\begin{prop}\label{thm:equality-spectral}
Let $A\in\BH$ be normal and let $r\geq1$. Then equality holds in
Proposition~\ref{thm:spectral} for $\theta=\frac12$. More precisely,
\[
w^{2r}(A)
=
\frac14\norm{|A|^{2r}+|A^*|^{2r}}
+
\frac12\,r\!\left(|A|^r|A^*|^r\right).
\]
\end{prop}

\begin{proof}
Since $A$ is normal, we have $|A|=|A^*|$ and $w(A)=\norm{A}.$
Therefore,
\[
\frac14\norm{|A|^{2r}+|A^*|^{2r}}
=
\frac14\norm{2|A|^{2r}}
=
\frac12\norm{A}^{2r}.
\]
Moreover,
\[
|A|^r|A^*|^r=|A|^{2r}.
\]
Since $|A|^{2r}$ is positive,
\[
r\!\left(|A|^r|A^*|^r\right)
=
r(|A|^{2r})
=
\norm{|A|^{2r}}
=
\norm{A}^{2r}.
\]
Hence
\[
\begin{aligned}
\frac14\norm{|A|^{2r}+|A^*|^{2r}}
+
\frac12\,r\!\left(|A|^r|A^*|^r\right)
=
\frac12\norm{A}^{2r}
+
\frac12\norm{A}^{2r}
=
\norm{A}^{2r}
=
w^{2r}(A).
\end{aligned}
\]
Thus equality holds in Proposition~\ref{thm:spectral} for
$\theta=\frac12$.
\end{proof}

\medskip

We conclude this section by considering the sharpness of the block-operator
estimate. The following result shows that the estimate in
Theorem~\ref{thm:block} is attained for a natural class of symmetric
off-diagonal operator matrices generated by normal operators. Let us define \[
T=
\begin{pmatrix}
0 & A\\
A & 0
\end{pmatrix}
\]
on $\mathcal H\oplus\mathcal H$.

\begin{prop}\label{thm:block-equality}
Let $A\in\BH$ be normal, let $r\geq1.$
Then,
\[
w^{2r}(T)
=
\frac14\norm{|A|^{2r}+|A^*|^{2r}}
+
\frac12\,w\!\left(|A|^r|A^*|^r\right)
=
\norm{A}^{2r}.
\]
\end{prop}

\begin{proof}
Consider the unitary operator
\[
U=
\frac{1}{\sqrt{2}}
\begin{pmatrix}
I&I\\
I&-I
\end{pmatrix}.
\]
A direct computation gives
\[
U^*TU
=
\begin{pmatrix}
A&0\\
0&-A
\end{pmatrix}.
\]
Since the numerical radius is invariant under unitary equivalence,
\[
w(T)
=
w\!\left(
\begin{pmatrix}
A&0\\
0&-A
\end{pmatrix}
\right)
=
w(A).
\]
As $A$ is normal $w(A)=\norm{A}$
and hence $w^{2r}(T)=\norm{A}^{2r}.$ On the other hand, taking $B=A$ and $\theta=\frac12$ in
Theorem~\ref{thm:block} yields
\[
w^{2r}(T)
\leq
\frac14\norm{|A|^{2r}+|A^*|^{2r}}
+
\frac12\,w\!\left(|A|^r|A^*|^r\right).
\]
Since $A$ is normal, $|A|=|A^*|$. Therefore,
\[
\frac14\norm{|A|^{2r}+|A^*|^{2r}}
=
\frac12\norm{A}^{2r},
\]
and
\[
w\!\left(|A|^r|A^*|^r\right)
=
w(|A|^{2r})
=
\norm{A}^{2r}.
\]
Consequently,
\[
\frac14\norm{|A|^{2r}+|A^*|^{2r}}
+
\frac12\,w\!\left(|A|^r|A^*|^r\right)
=
\norm{A}^{2r}
=
w^{2r}(T).
\]
Thus equality holds in Theorem~\ref{thm:block} for
$\theta=\frac12$.
\end{proof}

\medskip

Finally, we summarize the sharpness phenomena by our results.

\begin{rem}
The results of this section show that the proposed estimates are sharp on
natural classes of operators. In particular, normal operators attain equality
in the main estimate at $\theta=\frac12$, and the corresponding equality is
inherited by the symmetric off-diagonal block operators considered above.
Determining necessary and sufficient conditions for equality for general
operators, as well as understanding the dependence of equality on the
parameter $\theta$, remains an interesting problem.
\end{rem}

\section{Numerical examples illustrating sharpness} 

\begin{example}
Let
\[
A=
\begin{pmatrix}
0&2\\
0&0
\end{pmatrix}.
\]
Then
\[
|A|=
\begin{pmatrix}
0&0\\
0&2
\end{pmatrix},
\qquad
|A^*|=
\begin{pmatrix}
2&0\\
0&0
\end{pmatrix}.
\]
Since $A$ is a nilpotent operator of order two,
\[
w(A)=1,
\qquad
\norm{A}=2.
\]

Taking $r=1$ and $\theta=\frac12$ in
Theorem~\ref{thm:main}, we obtain
\[
w^2(A)
\leq
\frac14\norm{|A|^2+|A^*|^2}
+
\frac12 w(|A||A^*|).
\]
Now
\[
|A|^2+|A^*|^2
=
\begin{pmatrix}
4&0\\
0&4
\end{pmatrix}
=4I,
\]
while
\[
|A||A^*|=0.
\]
Consequently,
\[
w^2(A)
\leq
\frac14(4)
+
\frac12(0)
=1.
\]
Since $w(A)=1$, equality holds. Thus the estimate in
Theorem~\ref{thm:main} can be sharp even for a non-normal operator.
\end{example}

\medskip

\begin{example}
Let
\[
A=
\begin{pmatrix}
2&0\\
0&1
\end{pmatrix}.
\]
Then $A$ is positive and normal, and hence $|A|=|A^*|=A.$ Moreover,
$w(A)=\norm{A}=2.$
Let $r\geq1$ and choose $\theta=\frac12$. Then
\[
|A|^{2r\theta}|A^*|^{2r(1-\theta)}
=
|A|^r|A^*|^r
=
A^{2r}.
\]
Therefore, Proposition~\ref{thm:main} gives
\[
w^{2r}(A)
\leq
\frac14\norm{|A|^{2r}+|A^*|^{2r}}
+
\frac12 w\!\left(|A|^r|A^*|^r\right).
\]
Since $|A|^{2r}=|A^*|^{2r}
=
\begin{pmatrix}
2^{2r}&0\\
0&1
\end{pmatrix},$ we obtain
\[
\frac14\norm{|A|^{2r}+|A^*|^{2r}}
=
\frac12\,2^{2r},\;\;\text{and}\quad \frac12w\!\left(|A|^r|A^*|^r\right)
=
\frac12w(A^{2r})
=
\frac12\,2^{2r}.
\]

Hence
\[
w^{2r}(A)
\leq
2^{2r}.
\]
On the other hand, $w^{2r}(A)=2^{2r},$ and therefore equality holds. Furthermore, since $A^{2r}$ is positive,
\[
r(A^{2r})=w(A^{2r})=\norm{A^{2r}}=2^{2r}.
\]
Thus equality is also attained in the corresponding spectral-radius
estimate of Theorem~\ref{thm:spectral} at $\theta=\frac12$.
This illustrates the sharpness results obtained in Section~5.
\end{example}

\medskip

\begin{example}
Let $A=
\begin{pmatrix}
1&2\\
0&1
\end{pmatrix}.$ Then $A$ is invertible and non-normal. Indeed,
\[
A^*A
=
\begin{pmatrix}
1&2\\
2&5
\end{pmatrix},
\qquad
AA^*
=
\begin{pmatrix}
5&2\\
2&1
\end{pmatrix},
\]
so that $A^*A\neq AA^*$. For $r=1$, define
\[
\Phi_{A,1}(\theta)
=
\frac14
\norm{
|A|^{4\theta}
+
|A^*|^{4(1-\theta)}
}
+
\frac12
w\!\left(
|A|^{2\theta}|A^*|^{2(1-\theta)}
\right),
\qquad 0\leq\theta\leq1.
\]
By Theorem~\ref{thm:main}, $w^2(A)\leq\Phi_{A,1}(\theta),\quad$
for every $\theta\in[0,1]$. At the symmetric value $\theta=\frac12$, this becomes
\[
w^2(A)
\leq
\frac14\norm{|A|^2+|A^*|^2}
+
\frac12w(|A||A^*|),
\]
which is precisely the corresponding symmetric estimate.
Thus, even for this non-normal operator, Theorem~\ref{thm:main}
provides a family of admissible upper bounds parameterized by
$\theta$, whose optimal value may be obtained by minimizing
$\Phi_{A,1}(\theta)$ over $[0,1]$.
\end{example}

\medskip

\begin{example}
Let
\[
T=
\begin{pmatrix}
0&A\\
A&0
\end{pmatrix},
\qquad
A=
\begin{pmatrix}
1&1\\
0&1
\end{pmatrix}.
\]
Since $A$ is non-normal, the block operator $T$ is also non-normal.
Moreover, $T$ is unitarily equivalent to $A\oplus(-A)$, and hence
\[
w(T)=w(A).
\]
For the matrix $A$ above, its numerical range is the disk centered at $1$
with radius $\frac12$, so that $w(A)=\frac32.$ Therefore, $w^2(T)=\frac94.$ We now apply Theorem~\ref{thm:block} with $r=1$ and
$\theta=\frac12$. Since
\[
A^*A=
\begin{pmatrix}
1&1\\
1&2
\end{pmatrix},
\qquad
AA^*=
\begin{pmatrix}
2&1\\
1&1
\end{pmatrix},
\]
we have
\[
|A|
=
\frac1{\sqrt5}
\begin{pmatrix}
2&1\\
1&3
\end{pmatrix},
\qquad
|A^*|
=
\frac1{\sqrt5}
\begin{pmatrix}
3&1\\
1&2
\end{pmatrix}.
\]

It is easy to verify that 
\[
\norm{|A|^2+|A^*|^2}=5.
\]
Furthermore, $|A||A^*|
=
\frac15
\begin{pmatrix}
7&4\\
6&7
\end{pmatrix},$
and a direct calculation gives $w(|A||A^*|)=\frac{12}{5}.$ Hence Theorem~\ref{thm:block} yields
\[
\begin{aligned}
w^2(T)
&\leq
\frac14\norm{|A|^2+|A^*|^2}
+
\frac12 w(|A||A^*|)=
\frac54+\frac65
=
\frac{49}{20}.
\end{aligned}
\]
Since $\frac94\ < \frac{49}{20},$
the inequality is strict in this case. Thus, in contrast with the normal symmetric block operators considered in
Proposition~\ref{thm:block-equality}, the present non-normal example does
not attain equality at the symmetric parameter $\theta=\frac12$.
\end{example}

\begin{rem}
The preceding examples demonstrate that the estimates obtained above are
sharp, while also illustrating that their equality behavior depends
significantly on the structure of the operator. Normal operators attain
equality at the symmetric choice $\theta=\frac12$, whereas equality may also
occur for certain non-normal, non-invertible operators. In contrast, the
non-normal block-operator example above exhibits strict inequality. These
observations suggest that a complete characterization of the equality cases
for general operators requires a more detailed analysis of both the operator
structure and the choice of the parameter $\theta$.
\end{rem}

\section{Conclusion}

In this paper, we established a parameterized family of upper bounds for the
numerical radius of bounded linear operators using weighted expressions
involving $|A|$ and $|A^*|$. The symmetric choice $\theta=\frac12$ recovers
a known numerical-radius inequality, while optimization over $\theta$ provides
additional flexibility in the resulting estimates. We also obtained
corresponding results for $2\times2$ block operator matrices and identified
natural classes of operators for which the bounds are sharp. These results
suggest further investigation of optimal parameter choices and equality cases,
as well as extensions to other structured operator settings.

\section{Acknowledgment}  The authors are grateful to the Mohapatra Family Foundation and the College of Graduate Studies of the University of Central Florida for their support during this research. 

\section{Conflict of Interest} 
On behalf of all authors, the corresponding author states that there is no conflict of interest.

\end{document}